\documentclass[preprint, 12pt, english]{elsarticle}
\usepackage{amsmath}
\usepackage{latexsym, amssymb}
\usepackage{amsthm}

\newtheorem{thm}{Theorem}[section] %the resolution could also be [subsection]

\newtheorem{lem}[thm]{Lemma}

\newtheorem{rem}[thm]{Remark}

\def\lra{{\,\longrightarrow\,}}

\newcommand\operA[2]{{\if!#2!\operatorname{#1}\else{\operatorname{#1}_{#2}^{\phantom{I}}}\fi}} % To be used within Bdefs. Usage: $\operA{N}{K/F}$ produces $N_{K/F}$; $\operA{N}{}$ produces $N$.

\newcommand\Cref[1]{{Corollary~\ref{#1}}}%
\newcommand{\Trace}[1][]{\if!#1!\operatorname{Tr}\else{\operatorname{Tr}_{#1}^{\phantom{I}}}\fi} % Usage: $\Tr[K/F](a)$.

\long\def\forget#1\forgotten{{}} %

\def\({\left(}
\def\){\right)}

\newif\iffurther
\furtherfalse
\newif\ifXY % turns XY version on/off
\XYtrue     % Turn it on
\ifXY

\input xy
\input xyidioms.tex
\usepackage{xy}
\xyoption{all} %
\fi % For \ifXY

\usepackage{babel}

\journal{??}

\begin{document}

\begin{frontmatter}

\title{Chain Lemma for Biquaternion Algebras in Characteristic 2}

\author{Adam Chapman\corref{cor}\corref{cor2}}
\ead{adam1chapman@yahoo.com}
\cortext[cor2]{This work was supported by the U.S.-Israel Binational Science Foundation (grant no. 2010149)}
\cortext[cor]{The author is a PhD student of Prof. Uzi Vishne}
\address{Department of Mathematics, Bar-Ilan University, Ramat-Gan 52900,  Israel}

\begin{abstract}
In this paper, we prove that for a given biquaternion algebra over a field of characteristic two, one can move from one symbol presentation to another by at most three steps, such that in each step at least one entry remains unchanged. If one requires that in each step two entries remain the same then their number increases to fifteen. We provide even more basic steps that in order to move from one symbol presentation to another one needs to use up to forty-five of them.
\end{abstract}


\end{frontmatter}

\section{Introduction}
Every quaternion algebra over a field $F$ has a symbol presentation $$[\alpha,\beta)=F[x,y : x^2+x=\alpha, y^2=\beta, x y+y x=y]$$ if $F$ is of characteristic $2$, and $$(\alpha,\beta)=F[x,y : x^2=\alpha, y^2=\beta, x y=-y x]$$ if $F$ is of characteristic not $2$, for some $\alpha,\beta \in F$.

A biquaternion algebra is a tensor product of two quaternion algebras. If $[\alpha,\beta)$ and $[\gamma,\beta)$ are symbol presentations of the quaternion algebras, then $[\alpha,\beta) \otimes [\gamma,\delta)$ is a symbol presentation of the biquaternion algebra. By a chain lemma we mean a theorem that provides a list of basic steps with which one can obtain from one symbol presentation of the biquaternion algebra all the other symbol presentations.

An element $x$ in the algebra is called Artin-Schreier if it satisfies $x^2+x \in F$ and square-central if $x^2 \in F$.

The chain lemma for biquaternion algebras in characteristic not two was studied recently in \cite{Siv} and \cite{ChapVish2}.

In this paper we prove a similar chain lemma for biquaternion algebras in case of characteristic $2$.
We study it through quadruples of standard generators.
A quadruple of generators is $(x,y,z,u)$ such that $$x^2+x=\alpha, y^2=\beta, z^2+z=\gamma, u^2=\delta,$$ $$x y+y x=y, x z=z x, x u=u x, y z=z y, y u=u y, z u+u z=u$$ where $[\alpha,\beta) \otimes [\gamma,\delta)$ is the algebra under discussion.
The quadruple consists naturally of two pairs, $(x,y)$ and $(z,u)$.
We are not concerned with the order of the pairs, i.e. $(x,y,z,u)=(z,u,x,y)$.
Of course the order of the elements inside the pair is important, the first element corresponds to a separable field extension of the center and the second corresponds to an inseparable field extension. The first element is Artin-Schreier, and the second element is square-central.

We define the following steps on a quadruple of generators $(x,y,z,u)$:
\begin{itemize}
\item[$\Lambda_3$]: At most three generators are changed.
\item[$\Lambda_2$]: At most one generator is changed in each pair.
\item[$\Pi$]: At most one pair is changed.
\item[$\Omega_s$]: $x$ and $z$ are preserved and $y$ and $u$ are multiplied by $a+b (x+z)$ for some $a,b \in F$
\item[$\Omega_i$]: $y$ and $u$ are preserved and an element of the form $a y u$ is added to $x$ and $z$ for some $a \in F$.
\item[$\Omega_c$] : $y$ and $z$ are preserved and $x$ changes to $x+b y (1+b y)^{-1} z$ and $u$ changes to $(1+b y) u$ for some $b \in F$.
\item[$\Lambda_1$]: At most one generator is changed.
\end{itemize}

We prove that one can move from one quadruple of generators to another by a chain consisting of up to three steps of type $\Lambda_3$.
We prove further that every step of type $\Lambda_3$ can be replaced with up to three steps of type $\Pi$ and two steps of type $\Lambda_2$.
Furthermore, we prove that each step of type $\Lambda$ can be replaced with up to either three steps of type $\Lambda_1$ or two of type $\Lambda_1$ and one of type $\Omega_i$, $\Omega_s$ or $\Omega_c$.
Since $\Pi$ changes only one quaternion algebra, it is known that $\Pi$ can be replaced with up to three steps of type $\Lambda_1$. Consequently, in order to move from one quadruple of generators to another one needs to do up to $45$ steps, where at most $6$ of them are of type $\Omega_i$, $\Omega_s$ or $\Omega_c$ and all the rest are of type $\Lambda_1$.

The basic steps on the quadruples of generators can be easily translated to basic steps on the symbol presentations.

The $\Omega_s$ step changes $[\alpha,\beta) \otimes [\gamma,\delta)$ to $$[\alpha,(a^2+a b+b^2 (\alpha+\gamma)) \beta) \otimes [\gamma,(a^2+a b+b^2 (\alpha+\gamma)) \delta)$$ for some given $a,b \in F$.

The $\Omega_i$ step changes $[\alpha,\beta) \otimes [\gamma,\delta)$ to $$[\alpha+a^2 \beta \delta,\beta) \otimes [\gamma+a^2 \beta \delta,\delta)$$ for some given $a \in F$.

The $\Omega_c$ step changes $[\alpha,\beta) \otimes [\gamma,\delta)$ to $[\alpha+\frac{b^2 \beta \gamma}{1+b^2 \beta},\beta) \otimes [\gamma,\delta(1+b^2 \beta))$ for some $b \in F$.

The $\Lambda_1$ step changes one of the quaternion algebras $[\alpha,\beta)$ to either to $[\alpha,(a^2+a b+b^2 \alpha) \beta)$ or $[\alpha+a^2+a+b^2 \beta,\beta)$ for some $\alpha,\beta \in F$.

Throughout this paper, let $A$ be a fixed biquaternion division algebra over a field $F$ of characteristic two.

\section{Decomposition with respect to maximal subfields}

In this section we shall prove that if $A$ contains a maximal subfield, generated either by two Artin-Schreier elements or one Artin-Schreier and one square-central, then it decomposes as the tensor product of two quaternion algebras such that each of the generators is contained in a different quaternion algebra.

These lemmas will be used later on in this paper.

\begin{lem}\label{instep0}
If $x$ and $z$ are commuting Artin-Schreier elements then there exist some square-central elements $u$ and $y$ such that $(x,y,z,u)$ is a quadruple of generators.
\end{lem}

\begin{proof}
If $x$ and $z$ are commuting Artin-Schreier elements then $C_A(F[x])$ is a quaternion algebra containing $z$. This algebra contains some $q$ such that $q^2 \in F[x]$ and $z q+q z=q$.
The involution on $F[x,u]$ satisfying $x^*=x+1$ and $z^*=z$ extends to $A$. In particular, $q^* x=x q^*$, and therefore $q^* \in C_A(F[x])$.
If $q^*=q$ then by taking $u=q$, $u$ is square-central and $z u+u z=u$. Otherwise, we take $u=q+q^*$. In particular $A=A_0 \otimes F[z,u]$. $x$ is in the quaternion subalgebra $A_0$ and therefore there exists some square-central element $y \in A_0$ such that $x y+y x=y$.
\end{proof}

\begin{lem}\label{instep}
If $x$ is Artin-Schreier and $u$ is square-central then there exist some Artin-Schreier element $z$ and some square-central element $y$ such that $(x,y,z,u)$ is a quadruple of generators.
\end{lem}

\begin{proof}
If $x$ is Artin-Schreier and $u$ is a square-central element commuting with $x$ then $C_A(F[x])$ is a quaternion algebra containing $u$. This algebra contains some $q$ such that $q^2+q \in F[x]$ and $q u+u q=u$.
The involution on $F[x,u]$ satisfying $x^*=x+1$ and $u^*=u$ extends to $A$. In particular, $q^* x=x q^*$, and therefore $q^* \in C_A(F[x])$.

For some $\beta \in F$, $u^2=\beta$.
Write $\mu=q (a+b u) q^*$ for some unknown $a,b \in F$.
Since $q+q^*$ is symmetric with respect to $*$ and commutes with $u$, $q+q^*=c+d u$ for some fixed $c,d \in F$.
Obviously $\mu^*=\mu$.
We want $\mu u+u \mu=u$.
It is a straight-forward calculation to see the condition becomes $1=a+a c+b d \beta+(a d+b c) u$.
Consequently, we want the following system to be satisfied:
\begin{eqnarray*}
1 & = & (c+1) a+d \beta b\\
0 & = & d a+c b
\end{eqnarray*}
This system has a solution, unless $c (c+1)=d^2 \beta$.

If $c (c+1) \neq d^2 \beta$ then by taking $z=q (a+b u) q^*$ where $a,b$ is a solution to the system above, $z$ is Artin-Schreier and $z u+u z=u$.

If $c (c+1)=d^2 \beta$ then $(q^*)^2+q^*=(q+c+d u)^2+(q+c+d u)=q^2+c^2+d^2 \beta+d u+q+c+d u=q^2+q$. This means that $q^2+q$ is invariant under $*$, and therefore $q^2+q \in F$. In this case we will take $z=q$.

All in all, one can find an Artin-Schreier element $z$ such that $z u+u z=u$ and $x z=z x$,
which means that $A=A_0 \otimes F[z,u]$. $x$ is in the quaternion subalgebra $A_0$ and therefore there exists some square-central element $y \in A_0$ such that $x y+y x=y$.
\end{proof}

\section{A chain consisting of steps of type $\Lambda_3$}

In this section we will show that every two generating quadruples are connected by a chain of up to three steps of type $\Lambda_3$.

\begin{lem}\label{E1lem}
For any two Artin-Schreier elements $x,z$, if they do not commute then the subalgebra $F[x,z]$ is a quaternion algebra, whose center is either $F$ or a quadratic extension of it.
\end{lem}

\begin{proof}
There exist $a,b \in F$ such that $x^2+x=a$ and $z^2+z=b$.
Let $r=x z+z x$, $t=x z+z x+z=r+z$.
It is easy to see that $x r+r x=r$, and $x t+t x=0$.

Since $z=r+t$, $z^2+z+b=r^2+t^2+r t+t r+r+t+b=0$.
Therefore $(z^2+z+b) x+x (z^2+z+b)=r t+t r+r=0$.

Since $s=x+t$ commutes with $x,t,r$, it is in the center of $F[x,z]$.
The elements $x$ and $r$ generate a quaternion algebra over the center of $F[x,z]$, and since $t$ differs from $x$ by a central element, $F[x,z]$ is a quaternion algebra over its center.

Since $F[x,z]$ is a subalgebra of a biquaternion algebra, it cannot be the entire algebra, and therefore its center is either $F$ or a quadratic field extension of $F$.
\end{proof}

\begin{lem}\label{E1}
If $x$ and $z$ are not commuting Artin-Schreier elements then there exists some $w \in V$ which is either Artin-Schreier or square-central and commutes with them both.
\end{lem}

\begin{proof}
If the center of $F[x,z]$ is a quadratic extension of $F$ then it is generated by some $w \in V$, and that finishes the proof.
Otherwise, according to Lemma \ref{E1lem} the center of $F[x,z]$ is $F$ and $A=F[x,z] \otimes F[w,u : w^2+w=c,u^2=d,w u+u w=u]$ for some $c,d \in F$, and this also finishes the proof.
\end{proof}

\begin{thm}\label{Onesteps}
Every two quadruples of generators are connected by a chain of up to three steps of type $\Lambda_3$.
\end{thm}

\begin{proof}
Let $(x,y,z,u)$ and $(x',y',z',u')$ be two quadruples of generators.
If $x$ and $x'$ are not commuting then according to Lemma \ref{E1} there exists some $w$ which is either Artin-Schreier or square-central commuting with $x$ and $x'$.

If $w$ is Artin-Schreier then according to Lemma \ref{instep0} there exist $s,t \neq 0$ such that $(x,s,w,t)$ is a quadruple of generators.

Similarly, there exist some $s',t'$ such that $(x',s',w,t')$ is a quadruple of generators.

Consequently, there is a chain $$(x,y,z,u) \stackrel{\Lambda_3}{\lra} (x,s,w,t) \stackrel{\Lambda_3}{\lra} (x',s',w,t') \stackrel{\Lambda_3}{\lra} (x',y',z',u').$$

If $w$ is square-central then according to Lemma \ref{instep} there exist $s,t \neq 0$ such that $(x,s,t,w)$ is a quadruple of generators.

Similarly, there exist some $s',t'$ such that $(x',s',t',w)$ is a quadruple of generators.

Consequently, there is a chain $$(x,y,z,u) \stackrel{\Lambda_3}{\lra} (x,s,t,w) \stackrel{\Lambda_3}{\lra} (x',s',t',w) \stackrel{\Lambda_3}{\lra} (x',y',z',u').$$

If $x$ and $x'$ are commuting then according to Lemma \ref{instep0} there exist $s,t \neq 0$ such that $(x,s,x',t)$ is a quadruple of generators.

Consequently, there is a chain $$(x,y,z,u) \stackrel{\Lambda_3}{\lra} (x,s,x',t) \stackrel{\Lambda_3}{\lra} (x',y',z',u').$$
\end{proof}

\section{Replacing a step of type $\Lambda_3$ with steps of types $\Pi$ and $\Lambda_2$}

In this section we shall show how a step of type $\Lambda_3$ can be obtained by up to three steps of type $\Pi$ and two of type $\Lambda_2$.

\begin{lem}\label{notcomm}
If $y$ and $y'$ are two non-commuting square-central elements in $A$ then $F[y,y']$ is a quaternion algebra either over $F$ or over a quadratic extension of $F$. In particular, there exists either an Artin-Schreier element or a square-central element that commutes with both of them.
\end{lem}

\begin{proof}
Let $t=y y'+y' y$ and $r=y y'+y' y+y'$. It is easy to see that $y t=t y$, $y' t=t y'$ and $y r+r y=t$. In particular $t$ is in the center of $F[y,y']$.
If $t=0$ then $y'=r$ and $y'$ commutes with $y$, but we assumed the contrary, and so $t \neq 0$.
For similar reasons $r \neq 0$.

Let $q=y r t^{-1}$. It is a straight-forward calculation to see that $q \in V$ and $q r+r q=r$. Consequently $q$ and $r$ generate a quaternion algebra over the center of $F[y,y']$. Since this center contains $t$, it is easy to see that $y$ and $y'$ belong to that quaternion algebra, and therefore $F[y,y']=K[q,r]$ where $K=Z(F[y,y'])$.
Since it is a subalgebra of a biquaternion algebra over $F$, its center can be either $F$ or a quadratic extension of $F$. In both cases there exists either an Artin-Schreier element or a square-central element that commutes with both $y$ and $y'$.
\end{proof}

\begin{thm} \label{TwoSteps}
Every step of type $\Lambda_3$ can be achieved by at most three steps of type $\Pi$ and two of type $\Lambda_2$.
\end{thm}

\begin{proof}
A step of type $\Lambda_3$ preserves either an Artin-Schreier generator or a square-central generator.

Assume that it preserves an Artin-Schreier generator, i.e. $$(x,y,z,u)\stackrel{\Lambda_3}{\lra}(x,y',z',w').$$

If $y' \in F[x,y]$ then $$(x,y,z,u)\stackrel{\Lambda_1}{\lra}(x,y',z,u)\stackrel{\Pi}{\lra}(x,y',z',u').$$

Otherwise, if $y'$ commutes with $y$ then $$(x,y,z,u)\stackrel{\Pi}{\lra}(x,y,?,y y')\stackrel{\Lambda_2}{\lra}(x,y',?,y y')\stackrel{\Pi}{\lra}(x,y',z',u').$$

Assume that they do not commute.
According to Lemma \ref{notcomm}, there exists either an Artin-Schreier element or a square-central element $t$ commuting with both $y$ and $y'$.

If $\mu=x t+t x+t \not \in F$ then it is a straight-forward calculation to show that $\mu$ commutes with $x$, $y$ and $y'$, and so $\mu$ generates a quadratic extension in both $F[z,u]$ and $F[z',u']$. If it is separable then $$(x,y,z,u)\stackrel{\Pi}{\lra}(x,y,\mu,?)\stackrel{\Lambda_2}{\lra}(x,y',\mu,?)\stackrel{\Pi}{\lra}(x,y',z',u'),$$ and if inseparable then $$(x,y,z,u)\stackrel{\Pi}{\lra}(x,y,?,\mu)\stackrel{\Lambda_2}{\lra}(x,y',?,\mu)\stackrel{\Pi}{\lra}(x,y',z',u').$$

Otherwise, $t$ could be picked such that $\mu=0$ and then $x t+t x=t$, and therefore $t$ must be square-central. In this case \begin{eqnarray*}(x,y,z,u)&\stackrel{\Pi}{\lra}&(x,y,?,t y)\stackrel{\Lambda_2}{\lra}(x,t,?,t y)\\&\stackrel{\Pi}{\lra}&(x,t,?,t y')\stackrel{\Lambda_2}{\lra}(x,y',?,t y')\stackrel{\Pi}{\lra}(x,y',z',u').
\end{eqnarray*}

Assume that the initial $\Lambda_3$-step preserves a square-centarl generator, i.e. $$(x,y,z,u)\stackrel{\Lambda_3}{\lra}(x',y,z',w').$$

If $x' \in F[x,y]$ then $$(x,y,z,u)\stackrel{\Lambda_1}{\lra}(x',y,z,w)\stackrel{\Pi}{\lra}(x',y,z',w').$$

Otherwise, if $x'$ commutes with $x$ then $$(x,y,z,u)\stackrel{\Pi}{\lra}(x,y,x+x',?)\stackrel{\Lambda_2}{\lra}(x',y,x+x',?)\stackrel{\Pi}{\lra}(x',y,z',u').$$

Assume that they do not commute.
According to Lemma \ref{E1lem}, there exists either an Artin-Schreier element or a square-central element $t$ commuting with both $x$ and $x'$.

Let $\mu=t+y t y^{-1}$. This element commutes with $x$, $x'$ and $y'$. If $\mu \not \in F$ then $\mu$ generates a quadratic extension in both $F[z,u]$ and $F[z',u']$. If it is separable then $$(x,y,z,u)\stackrel{\Pi}{\lra}(x,y,\mu,?)\stackrel{\Lambda_2}{\lra}(x,y',\mu,?)\stackrel{\Pi}{\lra}(x,y',z',u'),$$ and if inseparable then $$(x,y,z,u)\stackrel{\Pi}{\lra}(x,y,?,\mu)\stackrel{\Lambda_2}{\lra}(x,y',?,\mu)\stackrel{\Pi}{\lra}(x,y',z',u').$$

If $\mu=0$ then $t$ commutes with $y$ and hence $t \in F[z,u]$. If $t$ is square-central then $$(x,y,z,u)\stackrel{\Pi}{\lra}(x,y,t,?)\stackrel{\Lambda_2}{\lra}(x',y,t,?)\stackrel{\Pi}{\lra}(x',y,z',u'),$$ and if Artin-Schreier then $$(x,y,z,u)\stackrel{\Pi}{\lra}(x,y,?,t)\stackrel{\Lambda_2}{\lra}(x',y,?,t)\stackrel{\Pi}{\lra}(x',y,z',u').$$

If $\mu \in F^\times$ then $(\mu^{-1} t) y+y (\mu^{-1} t)=y$, which means that $\mu^{-1} t$ is Artin-Schreier, but $t$ was either Artin-Schreier or square-central to begin with, and therefore $\mu=1$.
In this case, $t+x,t+x' \not \in F$, because otherwise $x$ and $x'$ commute, and we assumed that they do not.
Now, $t+x$ commutes with both $x$ and $y$, which means that it generates a quadratic extension of $F$ inside $F[z,u]$, which means that either $a (t+x)$ is Artin-Schreier for some $a \in F^\times$ or $t+x$ is square-central.
Similarly, $t+x'$ commutes with both $x'$ and $y$, which means that it generates a quadratic extension of $F$ inside $F[z',u']$, which means that either $a' (t+x')$ is Artin-Schreier for some $a' \in F^\times$ or $t+x'$ is square-central.

If $a (t+x)$ and $a' (t+x')$ are Artin-Schreier then we have \begin{eqnarray*}(x,y,z,u)&\stackrel{\Pi}{\lra}&(x,y,a(t+x),?)\stackrel{\Lambda_2}{\lra}(t,y,a(t+x),?)\\&\stackrel{\Pi}{\lra}&(t,y,a'(t+x'),?)\stackrel{\Lambda_2}{\lra}(x',y,a'(t+x'),?)\stackrel{\Pi}{\lra}(x',y,z',u').
\end{eqnarray*}

If $t+x$ and $t+x'$ are square-central then we have
\begin{eqnarray*}(x,y,z,u)&\stackrel{\Pi}{\lra}&(x,y,?,t+x)\stackrel{\Lambda_2}{\lra}(t,y,?,t+x)\\&\stackrel{\Pi}{\lra}&(t,y,?,t+x')\stackrel{\Lambda_2}{\lra}(x',y,?,t+x')\stackrel{\Pi}{\lra}(x',y,z',u').\end{eqnarray*}

If $a(t+x)$ is Artin-Schreier and $t+x'$ is square central then we have
\begin{eqnarray*}(x,y,z,u)&\stackrel{\Pi}{\lra}&(x,y,a(t+x),?)\stackrel{\Lambda_2}{\lra}(t,y,a(t+x),?)\\&\stackrel{\Pi}{\lra}&(t,y,?,t+x')\stackrel{\Lambda_2}{\lra}(x',y,?,t+x')\stackrel{\Pi}{\lra}(x',y,z',u').\end{eqnarray*}

The case of square-central $t+x$ and Artin-Schreier $a' (t+x')$ is essentially the same as the last one.
\end{proof}

\begin{rem}\label{twoentries}
As a result, every two quadruples of generators are connected by a chain of up to $9$ steps of type $\Pi$ and $6$ steps of type $\Lambda_2$.
\end{rem}

\section{Replacing a step of type $\Lambda_2$ with steps of types $\Omega_i$, $\Omega_s$, $\Omega_c$ and $\Lambda_1$}

I this section we shall show how a step of type $\Lambda_2$ can be obtained by up to three steps, one of which can be of type $\Omega_i$, $\Omega_s$ or $\Omega_c$ and the others are of type $\Lambda_1$. Since $\Pi$ can be obtained by up to three steps of type $\Lambda_1$, it means that every two quadruples of generators are connected by a chain of up to $45$ steps, where up to $6$ of them are of type $\Omega_i$, $\Omega _s$ or $\Omega_c$ and the rest are of type $\Lambda_1$.

\begin{lem}
If a step of type $\Lambda_2$ preserves two inseparable generators, i.e. $(x,y,z,u) \stackrel{\Lambda_2}{\lra} (x',y,z',u)$ then it can be achieved by at most two steps of type $\Lambda_1$ and one of type $\Omega_i$.
\end{lem}

\begin{proof}
The element $x z'+z' x+z'$ is nonzero because $(x z'+z' x+z') u+u (x z'+z' x+z')=u$. Consequently, $(x,y,x z'+z' x+z',u)$ is a quadruple of generators.
Similarly, $(x z'+z' x+x,y,z',u)$ is a quadruple of generators.
One can therefore do the following steps: $$(x,y,z,u) \stackrel{\Lambda_1}{\lra} (x,y,x z'+z' x+z',u) \stackrel{\Omega_i}{\lra} (x z'+z' x+x,y,z',u) \stackrel{\Lambda_1}{\lra} (x',y,z',u).$$

The element $r=x z'+z' x$ was added in the middle step to the Artin-Schreier generators. This element commutes with $y$ and $u$, and therefore it is in $F[u,y]$ and consequently of the form $a+b y+c u+d y u$. This element however also satisfies $x r+r x=z' r+r z'=r$. Hence, $a=b=c=0$.
\end{proof}

\begin{lem}
If a step of type $\Lambda_2$ preserves two Artin-Schreier generators, i.e. $(x,y,z,u)\stackrel{\Lambda_2}{\lra}(x,y',z,u')$ then it can be achieved by at most two steps of type $\Lambda_1$ and one of type $\Omega_s$.
\end{lem}

\begin{proof}
If $y u'+u' y=0$ then $y$ commutes with $u'$ and then one can do $$(x,y,z,u) \stackrel{\Lambda_1}{\lra} (x,y,z,u') \stackrel{\Lambda_1}{\lra} (x,y',z,u').$$

Otherwise, $y u'+u' y$ is square-central, and $(x,y,z,(y u'+u' y) y)$ is a quadruple of generators.
Similarly $(x,(y u'+u' y)^{-1} u',z,u')$ is a quadruple of generators.
One can therefore do $$(x,y,z,u) \stackrel{\Lambda_1}{\lra} (u,y,z,(y u'+u' y) y) \stackrel{\Omega_s}{\lra} (x,(y u'+u' y)^{-1} u',z,u') \stackrel{\Lambda_1}{\lra} (x,y',z,u').$$

In the middle step, the square-central generators were multiplied by $q=y^{-1} (y u'+u' y)^{-1} w'$. This element commutes with $x$ and $z$ and therefore $q \in F[x,z]$ and consequently of the form $a+b x+c z+d x z$. However, $q$ commutes with $y u'$, and therefore $d=0$ and $b=c$.
\end{proof}

\begin{lem}
If a step of type $\Lambda_2$ preserves one Artin-Schreier generator and one square-central generator, i.e. $(x,y,z,u)\stackrel{\Lambda_2}{\lra}(x',y,z,u')$ then it can be achieved by either at most three steps of type $\Lambda_1$ or at most two steps of type $\Lambda_1$ and one of type $\Omega_c$.
\end{lem}

\begin{proof}
If $x u'+u' x+u'$ is zero then $x u'+u' x=u'$.
Therefore one can do $(x,y,z,u) \stackrel{\Lambda_1}{\lra} (x,y,z,y u') \stackrel{\Lambda_1}{\lra} (x',y,z,y u') \stackrel{\Lambda_1}{\lra} (x',y,z,u')$.

Let us assume $x u'+u' x+u' \neq 0$.
\begin{eqnarray*}
x (x u'+u' x+u')+(x u'+u' x+u') x &= &x^2 u'+x u' x+x u'+x u' x+u' x^2+u' x\\=(x+\alpha) u'+x u'+u' (x+\alpha)+u' x &=&x u'+\alpha u'+x u'+u' x+\alpha u'+u' x=0.
\end{eqnarray*}
Therefore $x$ commutes with $x u'+u' x+u'$. In fact, $(x,y,z,x u'+u' x+u')$ is a quadruple of generators, and in particular $x u'+u' x+u'$ is square-central.
Now $$x (x u'+u' x+u')^{-1}(x u'+u' x)+(x u'+u' x+u')^{-1}(x u'+u'x) x=(x u'+u' x+u')^{-1}(x u'+u' x),$$ and $(x u'+u' x+u')^{-1}(x u'+u' x)$ commutes with $z$ and $x u'+u' x+u'$, and therefore $(x u'+u' x+u')^{-1}(x u'+u' x)=b y+c x y$ for some $b,c \in F$, but $(x u'+u' x+u')^{-1}(x u'+u' x)$ also commutes with $x u'+u' x$ while $y (x u'+u' x)+(x u'+u' x) y=0$ and $x y (x u'+u' x)+(x u'+u' x) x y=y (x u'+u' x)$ and therefore $c=0$. In particular $x u'+u' x=b (x u'+u' x+u') y$.

It is a straight-forward calculation to check that $$(x+b y (x u'+u' x+u') u'^{-1} z) u'+u' (x+b y (x u'+u' x+u') u'^{-1} z)=0,$$
$$(x+b y (x u'+u' x+u') u'^{-1} z) z+z (x+b y (x u'+u' x+u') u'^{-1} z)=0,$$ 
and so $(x+b y (x u'+u' x+u') u'^{-1} z,y,z,u')$ is a quadruple of generators too.

Therefore we have the chain
$$(x,y,z,u) \stackrel{\Lambda_1}{\lra} (x,y,z,x u'+u' x+u') \stackrel{\Omega_c}{\lra} (x+\beta y (x u'+u' x+u') u'^{-1} z,y,z,u') \stackrel{\Lambda_1}{\lra} (x',y,z,u').$$
\end{proof}

\begin{thm}
Every two quadruples of elements are connected by a chain of up to $45$ steps, of which up to $6$ are of type $\Omega_i$, $\Omega_s$ or $\Omega_c$ and the rest are of type $\Lambda_1$.
\end{thm}

\section*{Acknowledgements}
I would like to thank Prof. Jean-Pierre Tignol for the helpful discussions.

\section*{Bibliography}
\bibliographystyle{amsalpha}
\bibliography{bibfile}
\end{document}